\documentclass[english]{article}

\usepackage{amssymb,amsmath,amsfonts,amsthm,epsfig,pstricks,graphics,tikz}

\usepackage[width=0.85\textwidth ]{caption}

\newtheorem{theorem}{Theorem}[section]

\newtheorem{lemma}[theorem]{Lemma}
\usepackage{hyperref}

\title{\LARGE{\textbf{{ Playing a game of billiard with Fibonacci}}}}

\author{Daniel Jaud}
\date{}
\begin{document}
\maketitle
\flushbottom

\hrulefill\\ 

\begin{abstract}
By making use of the greatest common divisor's ($gcd$) properties we can highlight some connections between playing billiard inside a unit square and the Fibonacci sequence as well as the Euclidean algorithm. In particular by defining two maps $\tau$ and $\sigma$ corresponding to translations and mirroring we are able to rederive Lam\'{e}'s theorem and to equip it with a geometric interpretation realizing a new way to construct the golden ratio. Further we discuss distributions of the numbers $p,q\in \mathbb{N}$ with $gcd(q,p)=1$ and show that these also relate to the Fibonacci sequence.
\end{abstract}

\vspace*{0.7cm}
\hrulefill\\ 
\tableofcontents
\newpage

\section{Introduction and setup}\label{sec:Introduction}
The mathematical description of playing pool inside a square has been known for some time. In particular it is known that rational multiples of the angle $\pi$ lead to closed orbits. In this paper we  consider the set-up of a quadratic pool table of unit length\footnote{Quadratic tables of different size can simply be obtained by rescaling the $x-$ and $y-$directions by the same factor $L$, where $L$ is the length of a side.} with the four edges labelled $A,~B,~C$ and $D$ (see figure \ref{fig:setup}).

\begin{figure}[htb]
\centering
\begin{tikzpicture}
\draw[ultra thick] (0,0)--(2,0)--(2,2)--(0,2)--(0,0);
\node[left] at (0,0) {$A$};
\node[left] at (0,2) {$D$};
\node[right] at (2,0) {$B$};
\node[right] at (2,2) {$C$};

\end{tikzpicture}
\caption{Quadratic billiard table with four corners labelled as $A,~B,~C$ and $D$. \label{fig:setup}}
\end{figure}
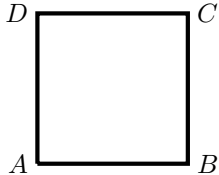

Assuming that our billiard ball is initially placed in corner $A$, we are interested in straight paths to another corner, where we allow the ball to scatter elastically when hitting one of the four boundaries. In our experiment we only consider the case of ideal reflection on boundaries, i.e. no loss of energy due to friction, spinning etc.. By the well known unfolding procedure \cite{billiard1,billiard2,Tabachnikov}(see exemplary figure \ref{fig:unfolding}) it  can easily be shown that (in the unfolded picture) the ball's trajectories correspond to straight lines, where the condition of hitting another corner is in an one to one correspondence for the slope $m$ to be given as a rationale number

\begin{equation}
m=\tan(\alpha)=\frac{q}{p}~~~\mbox{with}~~~p,q\in \mathbb{N}.
\end{equation}

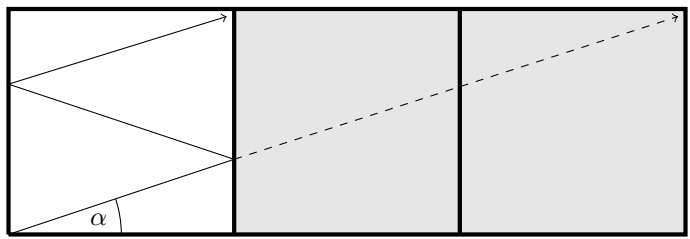
\begin{figure}[htb]
\centering
\begin{tikzpicture}
\fill[gray!20!white] (0,0)--(9,0)--(9,3)--(0,3)--(0,0);
\fill[white] (0,0)--(3,0)--(3,3)--(0,3)--(0,0);
\draw[ultra thick] (0,0)--(9,0)--(9,3)--(0,3)--(0,0);
\foreach \x in {3,6} \draw[ultra thick] (\x,0)--(\x,3);
\draw[->] (0,0)--(3,1)--(0,2)--(2.9,2.9);
\draw[dashed,->] (3,1)--(8.9,2.9);
\draw (1.5,0) arc (0:18.3:1.5) node[above] at (1.2,0) {$\alpha$};
\end{tikzpicture}
\caption{Unfolding of the original square trajectories for the ball which corresponds to straight lines. \label{fig:unfolding}}
\end{figure}

In particular $(p|q)$ can be identified with the point on the grid that is obtained by the unfolding procedure. As soon as the ball hits one of the corners the law of reflection is no longer applicable and we demand the ball to remain in this end position.

When this is the case there is one further restriction on the numbers $p$ and $q$, namely they additionally have to fulfil 

\begin{equation} \label{eq:gcd1}
gcd(p,q)=1.
\end{equation}

From a physical point of view equation \eqref{eq:gcd1} is necessary since if $gcd(q,p)=k$, with $k\in \mathbb{N}/\{1\}$, it holds that
\begin{equation}
m=\frac{q}{p}=\frac{k\cdot \tilde{q}}{k\cdot \tilde{p}}=\frac{\tilde{q}}{\tilde{p}}.
\end{equation}

This tells us that smaller numbers $(\tilde{p},\tilde{q})$ exist, leading exactly to the same slope, i.e. a ball's trajectory with smaller length than the trajectory corresponding to $(p,q)$. In particular, the ball will lie at rest in one of the corners provided we pick $(\tilde{p},\tilde{q})$.

\section{The $gcd$, maps and Fibonacci numbers}
The greatest common divisor $gcd$ of two numbers inhabits several properties. Here we will only focus on the following two

\begin{align}
gcd(a,b)&=gcd(b,a)~~~(asscociativity),\\
gcd(a,b)&=gcd(a,b~\mbox{mod}~a)~~~\mbox{for}~a>0.
\end{align}

In the last section we have seen that the numbers $p$ and $q$ for the initial slope of the ball's trajectory have also to fulfil the identity $gcd(p,q)=1$. Therefore, we can interpret the properties of the $gcd$ as maps acted on the slope $m=q/p$, leading to the definition of the two maps

\begin{align}
\sigma &:\frac{q}{p} \longmapsto \frac{p}{q} ~~~~~~~~ \leftrightarrow ~~~gcd(q,p)=gcd(p,q),\\
\tau &:\frac{q}{p} \longmapsto \frac{q}{p+q} ~~~ \leftrightarrow ~~~gcd(q,p)=gcd(q,p+q).
\end{align}

The action of $\sigma$ corresponds to a reflection of the straight line (point) in the unfolded picture along the diagonal with $m=1$, whereas $\tau$ corresponds to a parallel shift of the line's endpoints by $q$ in the $x-$direction (see figure \ref{fig:actions}).

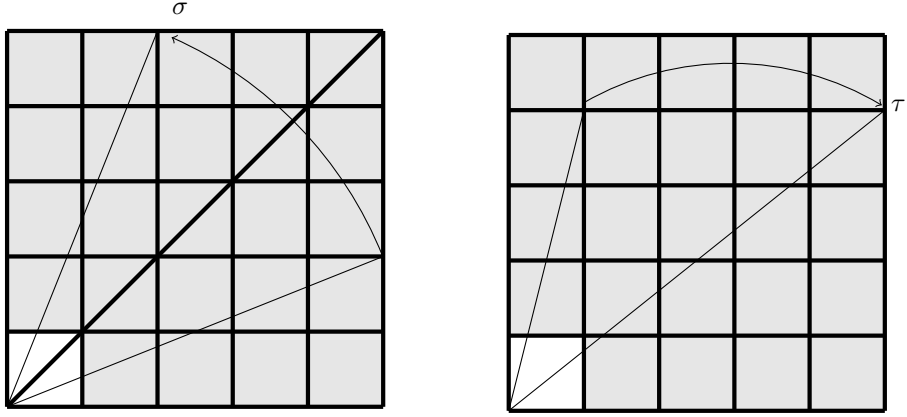
\begin{figure}[htb]
\centering
\begin{minipage}{0.45\textwidth}
\begin{tikzpicture}[scale=1]
\fill[gray!20!white] (0,0)--(5,0)--(5,5)--(0,5)--(0,0);
\fill[white] (0,0)--(1,0)--(1,1)--(0,1)--(0,0);

\foreach \x in {0,1,2,3,4,5} \draw[ultra thick] (\x,0)--(\x,5);
\foreach \x in {0,1,2,3,4,5} \draw[ultra thick] (0,\x)--(5,\x);
\draw[ultra thick] (0,0)--(5,5);
\draw (0,0)--(5,2);
\draw (0,0)--(2,5);
\draw[->] (21.8:5.38) arc (21.8:66:5.38) node[above] at (2.3,5.1) {$\sigma$};
\end{tikzpicture}
\end{minipage}
\hfill
\begin{minipage}{0.45\textwidth}
\begin{tikzpicture}[scale=1]
\fill[gray!20!white] (0,0)--(5,0)--(5,5)--(0,5)--(0,0);
\fill[white] (0,0)--(1,0)--(1,1)--(0,1)--(0,0);

\foreach \x in {0,1,2,3,4,5} \draw[ultra thick] (\x,0)--(\x,5);
\foreach \x in {0,1,2,3,4,5} \draw[ultra thick] (0,\x)--(5,\x);
\draw (0,0)--(1,4);
\draw (0,0)--(5,4);
\draw[->] (1,4.1) arc (120:59:3.9) node[right] {$\tau$};
\fill[white] (2,5.5) circle (0.1);
\end{tikzpicture}
\end{minipage}
\caption{Action of $\sigma$ and $\tau$ on a given straight line leading to an allowed corner in the unfolded picture. \label{fig:actions}}
\end{figure}

In particular, when starting with the slope $m=\frac{1}{1}$ all other allowed slopes can be obtained by multiple actions of $\sigma$ and $\tau$, e.g.

\begin{equation}
\frac{3}{5}=\frac{3}{2+3}=\tau \left(\frac{3}{2}\right) =\tau \sigma \left(\frac{2}{3}\right) =\tau \sigma \left(\frac{2}{1+2}\right)= \dots = \underbrace{\tau \sigma \tau \sigma \tau}_{\mbox{word}} \left( \frac{1}{1}\right).
\end{equation}

To each slope $m=\frac{q}{p}$ corresponds a unique word formed by $\tau$ and $\sigma$. The general form of these words is given by

\begin{equation}
\boxed{
\sigma^j \tau^{n_1}\sigma \tau^{n_2} \cdots \sigma \tau^{n_k}}~~, \label{eq:word}
\end{equation}

where $j=1$  if $q>p$ and $j=0$ for $q<p$. Hereby, the uniqueness directly follows from the non-commutativity of $\tau$ and $\sigma$, i.e.

\begin{equation}
[\tau,\sigma]\not=0.
\end{equation}



\begin{theorem}
For the number $N_\sigma (n)$ of $\sigma$'s within a word of length $n$ holds $N_\sigma(n)\leq \lfloor \frac{n}{2} \rfloor $.
\end{theorem}

\begin{proof}
The most general form of a word is given by equation \eqref{eq:word}. If the length $n$ of the word, i.e. the total number of $\tau$'s and $\sigma$'s, is fixed, the number $N_\sigma(n)$ of $\sigma$'s contained in the word becomes maximal if $n_i=1 ~\forall i $. If $n$ is even, there are as many $\tau$'s as there are $\sigma$'s, i.e. in this case $N_\sigma(n)=\frac{n}{2}$. If $n$ is odd, there always has to be one more $\tau$ compared to the number of $\sigma$'s, i.e. then holds

\begin{align}
N_\sigma(n)+\underbrace{N_\sigma(n)+1}_{=N_\tau(n)}&=n\\
\leftrightarrow N_\sigma(n)&=\frac{n-1}{2}.
\end{align}

Combining both results we find the following

\begin{equation}
N_\sigma(n)\leq \lfloor \frac{n}{2} \rfloor .
\end{equation}
\end{proof}

Next, we want to highlight the connection between the words \eqref{eq:word} and the Euclidean algorithm \cite{numbertheory,Moore87}. In order to do so we state:

\begin{lemma}\label{lemma:1step}
If $p=n_1\cdot q +r_1$, with $n_1\in \mathbb{N}$ and $r_1$ the rest of the division of $p$ with $q$ (assuming $p>q$), then it holds
\begin{equation}
\frac{q}{p}=\tau^{n_1}\sigma \left(\frac{r_1}{q}\right).
\end{equation}
\end{lemma}
The proof follows by direct application of the definitions of $\tau$ and $\sigma$ and is consequently clear. With this we can show:
\begin{theorem}
The word corresponding to $\frac{q}{p}$, with $q>p$, realizes the Euclidean algorithm. Furthermore, the number $N_\sigma (n)$ of $\sigma$'s within a word of length $n$ corresponds to the number of steps performed in the Euclidean algorithm.
\end{theorem}

\begin{proof}
Considering lemma \ref{lemma:1step} it represents the first step in the Euclidean algorithm (see e.g. \cite{numbertheory}). The additional slope $r_1/q$ can equally be rewritten via the second step in the Euclidean algorithm as
\begin{equation}
q=n_2\cdot r_1 + r_2 ~~ \leftrightarrow ~~ \tau^{n_2}\sigma \left(\frac{r_2}{r_1}\right).
\end{equation}

Following this procedure we end up with
\begin{equation}
\frac{q}{p}=\tau^{n_1}\sigma \tau^{n_2} \cdots \sigma \tau^{n_k}\left(\frac{1}{1}\right),
\end{equation}
and see that the number of $\sigma$'s within the word (here $N_\sigma=k-1$) counts the number of steps performed in the Euclidean algorithm.
\end{proof}

It is important to understand that only the number of $\sigma$'s inside the word, i.e. enclosed by $\tau$'s, are in one to one correspondence to the Euclidean algorithm since a single $\sigma$ which might be on the very left of the word simply corresponds to the inverse slope. 

\newpage

\section{Distribution of words with minimal/maximal number of $N_\sigma$}

In this section we want to focus on words with $N_\sigma(n)=1$, i.e. words that represent only one necessary step in the Euclidean algorithm, and words in which $N_\sigma(n)$ is maximal. We are going to start with $N_\sigma(n)=1$, i.e. the most general word that can be built takes the form

\begin{equation}
\tau^{n_1}\sigma \tau^{n_2},  \label{eq:word2}
\end{equation}

where $n_1+n_2+1=n$, $n_1\in \mathbb{N}_0$ and $n_2 \in \mathbb{N}$. Thus it holds

\begin{equation}
\tau^{n_1}\sigma \tau^{n_2} \left(\frac{1}{1}\right)=\frac{n_2+1}{n_1(n_2+1)+1}=\frac{q}{p}~~~~\rightarrow ~~~~\frac{q}{p}=\tau^{\frac{p-1}{q}}\sigma \tau^{q-1} \left(\frac{1}{1}\right).
\end{equation}

By graphing the allowed points $(p|q)$ it follows by construction (see figure \ref{fig:minimal}) that they can be connected by straight lines starting from the point $(1|0)$ on the grid.

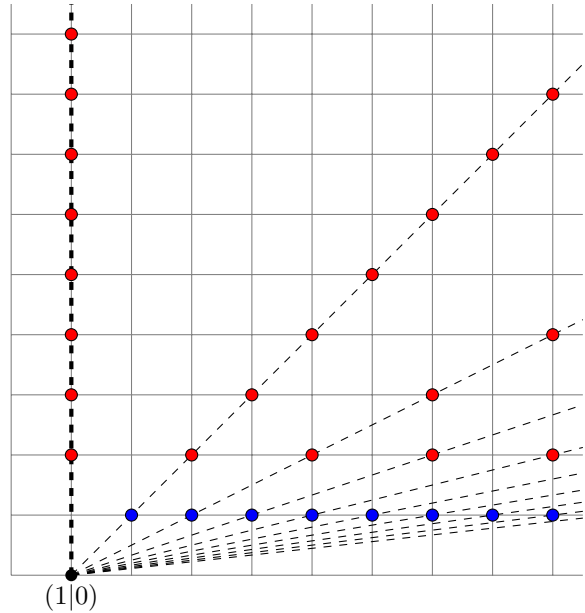
\begin{figure}[htb]
\centering
\begin{tikzpicture}[scale=0.8]
\draw[gray] (0,0) grid (9.5,9.5);
\draw[black,ultra thick, dashed] (1,0)--(1,9.5);
\draw[black,dashed] (1,0)--(9.5,8.5);
\draw[black,dashed] (1,0)--(9.5,4.25);
\draw[black,dashed] (1,0)--(9.5,2.8333);
\draw[black,dashed] (1,0)--(9.5,2.125);
\draw[black,dashed] (1,0)--(9.5,1.7);
\draw[black,dashed] (1,0)--(9.5,1.4166);
\draw[black,dashed] (1,0)--(9.5,1.214);
\draw[black,dashed] (1,0)--(9.5,1.0625);
\draw[black,dashed] (1,0)--(9.5,0.95);

\foreach \x in {2,3,4,5,6,7,8,9} \fill[blue,draw=black] (\x,1) circle (0.1); 
\foreach \x in {3,5,7,9} \fill[red,draw=black] (\x,2) circle (0.1) ; 
\foreach \x in {4,7} \fill[red,draw=black] (\x,3) circle (0.1);
\foreach \x in {5,9} \fill[red,draw=black] (\x,4) circle (0.1);
\foreach \x in {2,3,4,5,6,7,8,9} \fill[red,draw=black] (1,\x) circle (0.1);
\fill[red,draw=black] (6,5) circle (0.1);
\fill[red,draw=black] (7,6) circle (0.1);
\fill[red,draw=black] (8,7) circle (0.1);
\fill[red,draw=black] (9,8) circle (0.1);

\fill[black] (1,0) circle (0.1) node[below] {$(1|0)$};
\end{tikzpicture}
\caption{Blue points correspond to the words $\tau^{n_2}$. Red points correspond to words of the form $\tau^{n_1}\sigma \tau^{n_2}$ forming straight lines of slope $\tilde{m}=\frac{1}{n_1}$ with respect to the starting point $(1|0)$. \label{fig:minimal}}
\end{figure}

For the slope $\tilde{m}$ of these lines with respect to the point $(1|0)$ holds

\begin{equation}
\tilde{m}=\frac{1}{n_1}.
\end{equation}

In particular, we immediately know that when drawing a straight line with the starting point $(1|0)$ and slope $\tilde{m}=\frac{1}{n_1}$, where $n_1\in \mathbb{N}$, all the intersection points of the line and the grid represent the coordinates of two numbers $p$ and $q$ with $gcd(p,q)=1$. Further words with $N_\sigma >1$ that minimally differ from the words defined in \eqref{eq:word2} will accumulated near to one of the straight lines of slope $\tilde{m}$ leading to a distribution of points $p$ and $q$ with $gcd(p,q)=1$ with respect to the steps performed in the Euclidean algorithm.

Another example considers the special word $fib$ with $N_\tau(2)=N_\sigma(2)=1$ where the numbers of $\sigma$ and $\tau$ are equally distributed. The word thus takes the form

\begin{equation}
fib:=\sigma \tau.
\end{equation}

We have named the word $fib$ because when it acts on the slope $m=\frac{1}{1}$ it is deeply connected to Fibonacci numbers as stated in the following theorem.

\begin{theorem}\label{fibk}
Acting with $fib^k:=\overbrace{\sigma \tau \cdots \sigma \tau}^{k-times}$ on $m=\frac{1}{1}$ yields
\begin{equation}
fib^k\left(\frac{1}{1}\right)=\frac{F_{k+1}}{F_{k}}.
\end{equation}
Here $F_k$ is the $k-$th Fibonacci number with $F_0=F_1=1$.
\end{theorem}
\begin{proof}
We proof the theorem by total induction. For $k=1$ we have
\begin{equation}
fib^1\left(\frac{1}{1}\right)=\sigma\tau \left(\frac{1}{1}\right)=\frac{2}{1}=\frac{F_2}{F_1}.
\end{equation}
Now for $k\rightarrow k+1$ we find
\begin{equation}
fib^{k+1}\left(\frac{1}{1}\right)= fib \circ  fib^k \left(\frac{1}{1}\right) =\sigma \tau  \left(\frac{F_{k+1}}{F_{k}}\right)=\frac{F_{k}+F_{k+1}}{F_{k+1}}=\frac{F_{k+2}}{F_{k+1}}.
\end{equation}
\end{proof}
Note that the number of divisions in the Euclidean algorithm for $fib^k$ in theorem \ref{fibk} is exactly $N_\sigma=k$. Further, it is immediatelly followes by construction that $gcd(F_k,F_{k+1})=1$. We realize that this is nothing else than Lam\'{e}'s theorem in the formulation of D. Knuth \cite{Knuth}, though it is obtained by a different approach.
 
A direct consequence of theorem \ref{fibk} yields the following result
\begin{theorem}
In the limit of large word length $n$ it holds
\begin{equation}
\lim_{k\rightarrow \infty} fib^k \left(\frac{1}{1}\right) = \varphi, \label{eq:goldenration}
\end{equation}

where $\varphi=\frac{1+\sqrt{5}}{2}$ is the golden ratio.
\end{theorem}

\vspace*{0.5cm}
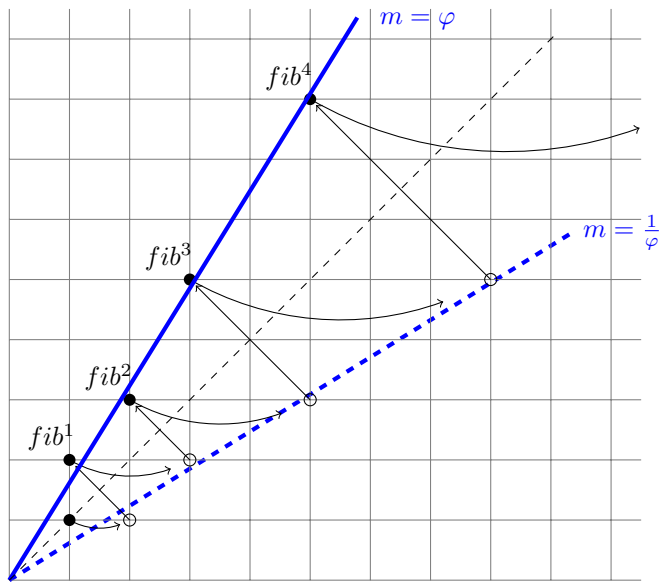
\begin{figure}[htb]
\centering
\begin{tikzpicture}[scale=0.8]
\draw[ultra thick,blue,dashed] (0,0)--(31.717:11) node[right] {\textbf{$m=\frac{1}{\varphi}$}};

\draw[gray] (0,0) grid (10.5,9.5);
\fill[black] (1,1) circle (0.1);
\draw (2,1) circle (0.1);
\draw[->] (1,1) arc (240:290:1);
\draw[->] (2,1)--(1.1,1.9);
\node[above] at (0.7,2) {$fib^1$};

\fill[black] (1,2) circle (0.1);
\draw (3,2) circle (0.1);
\draw[->] (1,2) arc (240:290:2);
\draw[->] (3,2)--(2.1,2.9);
\node[above] at (1.65,3) {$fib^2$};

\fill[black] (2,3) circle (0.1);
\draw (5,3) circle (0.1);
\draw[->] (2,3) arc (240:290:3);
\draw[->] (5,3)--(3.1,4.9);
\node[above] at (2.65,5) {$fib^3$};

\fill[black] (3,5) circle (0.1);
\draw (8,5) circle (0.1);
\draw[->] (3,5) arc (240:290:5);
\draw[->] (8,5)--(5.1,7.9);
\node[above] at (4.65,8) {$fib^4$};

\fill[black] (5,8) circle (0.1);
\draw[->] (5,8) arc (240:290:6.5);

\draw[blue,ultra thick] (0,0)--(58.2825:11);
\node[right,blue] at (6,9.3) {\textbf{$m=\varphi$}};

\draw[dashed] (0,0)--(9.1,9.1);
\end{tikzpicture}
\caption{Construction of the straight line with slope $m=\varphi$ by action of $fib^k$ on $\frac{1}{1}$. In addition, the line with slope $m=\frac{1}{\varphi}$ is shown. \label{fig:constructiongoldenratio}}
\end{figure}

The proof is clear by theorem \ref{fibk} and the general known limit for the quotient of Fibonacci numbers \cite{goldenration}. The advantage of \eqref{eq:goldenration} is that it shows us a new geometric point of view on how to construct the golden ration, where $\varphi$ is realized as the slope of the straight line through the origin (see for illustration figure \ref{fig:constructiongoldenratio}). Furthermore, if one considers the distribution of the $gcd$ of two numbers with respect to their number of steps performed in the Euclidean algorithm (see figure \ref{fig:distribution}) point which maximize the steps performed accumulate to the slope $m=\varphi$.

\begin{figure}[htb]
\centering
\includegraphics[scale=0.2]{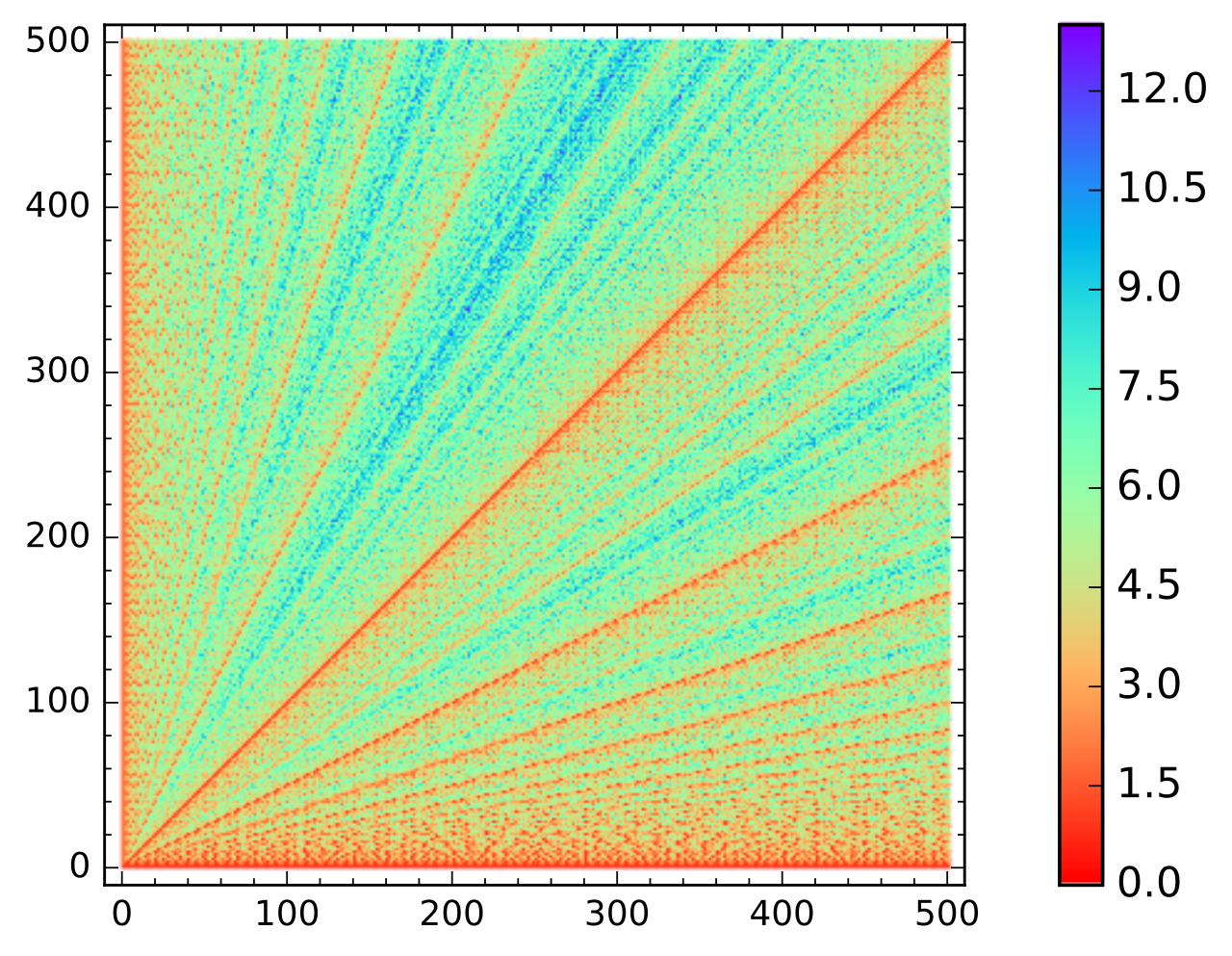}
\caption{Number of steps in the Euclidean algorithm for $gcd(p,q)$. Red points correspond to words with $N_\sigma=1$. Yellow and orange points correspond to small variations of the words which are represented by red points, whereas violet and blue points indicate more steps, i.e. $N_\sigma \rightarrow N_\tau$. The largest blue area follows the line $q=\varphi \cdot p$. The graphic is depicted from \cite{Wikipedia}.\label{fig:distribution}}
\end{figure}

Since we  have already seen that in case of the word $fib^k$ the number of $\sigma$'s is maximal for a given length of the word and that $N_\sigma(n)$ is in direct connection to the steps performed in the Euclidean algorithm, it is clear that points near the line $y=\varphi x$ correspond to small variations of the word $fib^k$.

It is interesting that in the context of billiard each word $fib^k$ corresponds to an allowed path for the billiard ball, whereas the limit 

\begin{equation}
\lim_{k\rightarrow \infty} fib^k \left(\frac{1}{1}\right)=\varphi
\end{equation}

represents an infinite trajectory of the billiard ball, see for a proof e.g. \cite{Tabachnikov}.

\section{Word lengths and Fibonacci numbers}

In the previous section we have introduced the maps $\sigma$ and $\tau$ and have shown that every allowed slope for the billiard ball trajectory is leading to another corner of the square can uniquely be obtained by acting a word of the form \eqref{eq:word} on the diagonal slope $m=\frac{1}{1}$. Now, we want to highlight another connection of words \eqref{eq:word} to Fibonacci numbers $F_n$, with $n\in \mathbb{N}_0$.

For this we consider all possible ways to build words of fixed length $n$, i.e.  we count how many times the maps $\tau$ and $\sigma$ appear within the word and how many ways exist to permute them in order to yield different words corresponding to distinct slopes or numbers fulfilling $gcd(q,p)=1$. For illustration the first possible words of length $n\leq 4$ are shown in a tree chart in figure \ref{fig:Baumdiagramm}.

\begin{figure}[htb]
\centering
\includegraphics[scale=0.5]{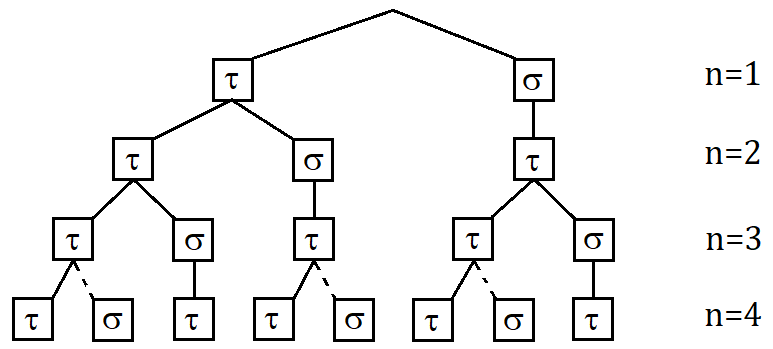}
\caption{The tree chart represents all possibilities for words with  $n\leq 4$. Words ending on $\sigma$ do not contribute and are pictured as dashed lines. \label{fig:Baumdiagramm}}
\end{figure}

For $n=1$ we see that we obtain one different word\footnote{note that $\sigma$ is not counted as distinct word since the action of $\sigma$ on $m=\frac{1}{1}$ is the identity map and is not leading to a new slope. In other words, we have defined the words \eqref{eq:word} not to start with the action of $\sigma$.}, for $n=2$ two, for $n=3$ three and for $n=4$ five. This is nothing else than the Fibonacci sequence.

Before proofing this statement, we start with a simpler result.

\begin{lemma}{\label{lemma:Fib1}}
The number of ways to build any word of length $n$ including $\sigma$ and $\tau$, assuming that $\sigma$ also can stand in first position, is given by $F_{n+1}$.
\end{lemma}

We can proof this by total induction.

\begin{proof}
Let $M_\sigma(n)$ denote the number of words ending on a $\sigma$ and as well for $M_\tau(n)$. For $n=1$ we have $M_\sigma(1)=M_\tau(1)=1$ and thus for the total number of different words $M_\sigma(1)+M_\tau(1)=2=F_2$. Performing the step $n\rightarrow n+1$ we have
\begin{align*}
M_\sigma(n+1)+M_\tau(n+1)&=\overbrace{M_\tau(n)}^{=M_\sigma(n+1)}+\underbrace{M_\sigma(n)+M_\tau(n)}_{=M_\tau(n+1)}\\
&=M_\tau(n)+F_{n+1}\\
&=M_\sigma(n-1)+M_\tau(n-1)+F_{n+1}\\
&=F_n+F_{n+1}=F_{n+2}.
\end{align*}

Here we have used the fact that every word ending on a $\tau$ can  give words ending on $\tau$ or $\sigma$in the next step, whereas words ending on $\sigma$ only yield words ending on $\tau$ in the following step since $\sigma^2=id$. 
\end{proof}

Using the result of lemma \ref{lemma:Fib1} we can show the form \eqref{eq:word} of our words:

\begin{theorem}
The number of different words of length $n$ is given by the Fibonacci number $F_{n}$, where $F_0=F_1=1$. 
\end{theorem}

\begin{proof}
As a direct consequence of lemma \eqref{lemma:Fib1} we have
\begin{align}
M_\sigma(n)&=F_{n-1}\\
M_\tau(n)&=F_n.
\end{align}

The number of words of length $n$, as defined in equation \eqref{eq:word}, is then given by $M_\tau(n)=F_n$. For a pictorial visualization see figure \ref{fig:Baumdiagramm2}.

\begin{figure}[htb]
\centering
\includegraphics[scale=0.32]{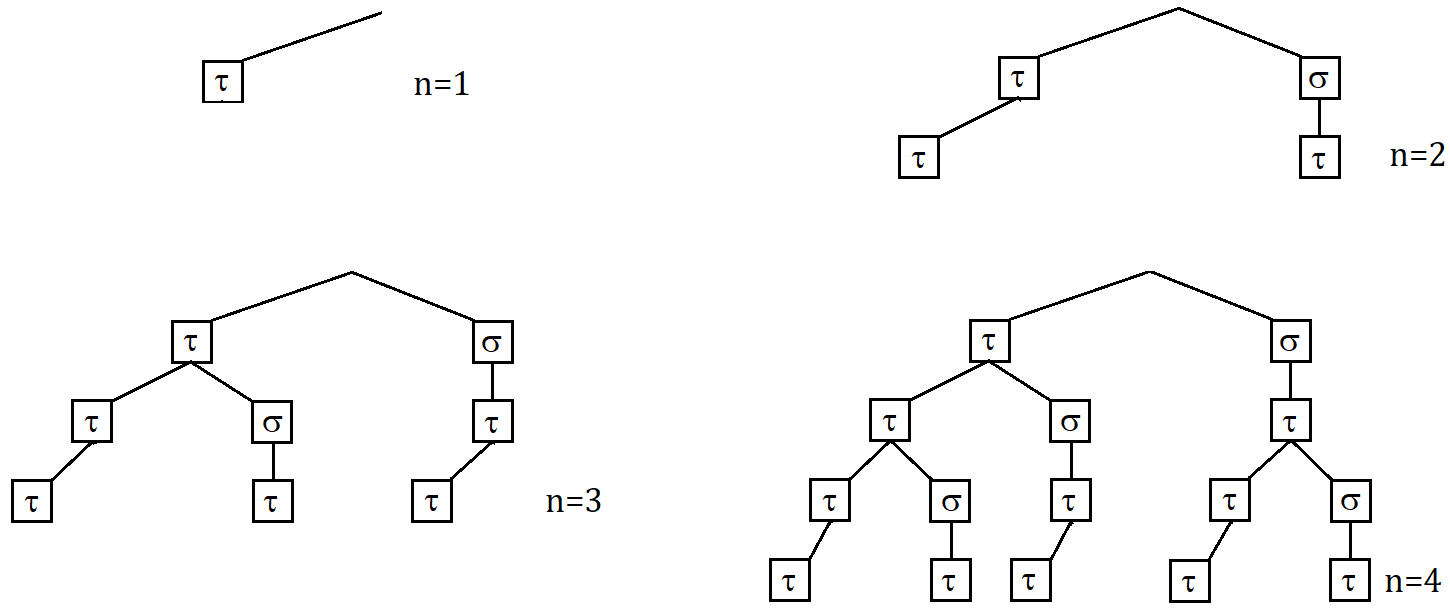}
\caption{Paths to build words of length $n\leq 4$. \label{fig:Baumdiagramm2}}
\end{figure}

\end{proof}

For a fixed $n$ the distribution of points that can be obtained by the words of length $n$, is shown in figure \ref{fig:distribution2}. From the graphic, again, it is interesting to see the relation to the Fibonacci numbers which are connected by mirroring points along the diagonal. In particular, by considering points only that lie below the diagonal we again identify the Fibonacci sequence for fixed $n$.

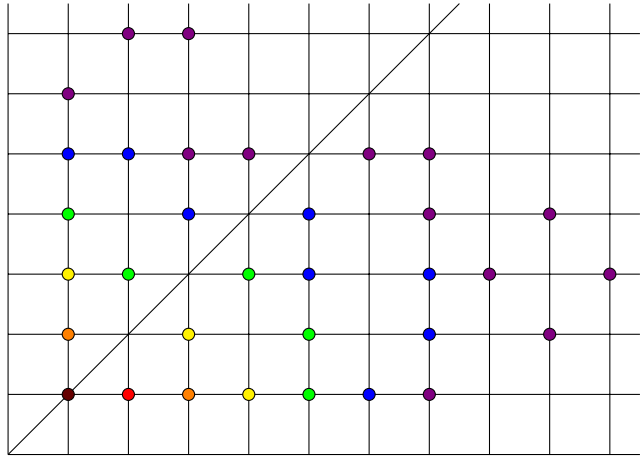
\begin{figure}[htb]
\centering
\begin{tikzpicture}[scale=0.8]
\draw (0,0) grid (10.5,7.5);

\draw(0,0)--(7.5,7.5);

\fill[red!40!black,draw=black] (1,1) circle (0.1); 

\fill[red,draw=black] (2,1) circle (0.1); 

\fill[orange,draw=black] (3,1) circle (0.1); 
\fill[orange,draw=black] (1,2) circle (0.1); 

\fill[yellow,draw=black] (4,1) circle (0.1); 
\fill[yellow,draw=black] (1,3) circle (0.1); 
\fill[yellow,draw=black] (3,2) circle (0.1); 

\fill[green,draw=black] (5,1) circle (0.1); 
\fill[green,draw=black] (5,2) circle (0.1); 
\fill[green,draw=black] (4,3) circle (0.1); 
\fill[green,draw=black] (2,3) circle (0.1); 
\fill[green,draw=black] (1,4) circle (0.1); 

\fill[blue,draw=black] (6,1) circle (0.1); 
\fill[blue,draw=black] (7,2) circle (0.1); 
\fill[blue,draw=black] (5,3) circle (0.1); 
\fill[blue,draw=black] (7,3) circle (0.1); 
\fill[blue,draw=black] (5,4) circle (0.1); 
\fill[blue,draw=black] (1,5) circle (0.1); 
\fill[blue,draw=black] (2,5) circle (0.1); 
\fill[blue,draw=black] (3,4) circle (0.1); 

\fill[violet,draw=black] (7,1) circle (0.1); 
\fill[violet,draw=black] (9,2) circle (0.1); 
\fill[violet,draw=black] (10,3) circle (0.1); 
\fill[violet,draw=black] (9,4) circle (0.1); 
\fill[violet,draw=black] (8,3) circle (0.1); 
\fill[violet,draw=black] (6,5) circle (0.1); 
\fill[violet,draw=black] (7,5) circle (0.1); 
\fill[violet,draw=black] (7,4) circle (0.1); 
\fill[violet,draw=black] (1,6) circle (0.1); 
\fill[violet,draw=black] (2,7) circle (0.1); 
\fill[violet,draw=black] (3,7) circle (0.1); 
\fill[violet,draw=black] (4,5) circle (0.1); 
\fill[violet,draw=black] (3,5) circle (0.1); 




\end{tikzpicture}

\caption{Distribution of all endpoints for words of fixed length. Red equals $n=1$, orange $n=2$, yellow $n=3$, green $n=4$, blue $n=5$ and violet $n=6$. \label{fig:distribution2}}
\end{figure}

At this point it is unclear whether there is some deeper relation to the distribution or if it can be used in some other context. Nevertheless, it is interesting that the Fibonacci sequence seems to be a manifest feature of the $gcd$.

\section{Conclusion and Outlook}\label{sec:conclusion}
By defining the two maps $\tau$ and $\sigma$ related to properties of the $gcd$ we have shown that the Fibonacci sequence is naturally contained in various aspects of words that can be built by the two maps. We have been able to rederive Lam\'{e}'s theorem and also to  equip it with a geometrical interpretation in terms of trajectories for a billiard ball rolling, by possible multiple reflections, from one corner to another. In addition, we proved that the number of words of length $n$ is also connected to the Fibonacci sequence, though the profit one possibly can gain from this fact is up to this point not clear. For future works it would be interesting to see whether the Fibonacci sequence is even further embedded in the framework of mathematical billiards, e.g. if it also appears in the case of a triangular (or more generally polygonal) tables or if the distribution of words of fixed length $n$ can be used to estimate the number of trajectories of length less or equal to some fixed number $L$. 



\appendix


\vspace*{3cm}
\begin{center}
\textit{A special thanks to Agnes Zauner for her enlightening humour.}
\end{center}

\end{document}